\documentclass[reqno]{amsart}
\usepackage{amssymb,setspace}
\usepackage{ifpdf}
\ifpdf
  \usepackage[hyperindex]{hyperref}
\else
  \expandafter\ifx\csname dvipdfm\endcsname\relax
    \usepackage[hypertex,hyperindex]{hyperref}
  \else
    \usepackage[dvipdfm,hyperindex]{hyperref}
  \fi
\fi
\theoremstyle{plain}
\newtheorem{thm}{Theorem}[section]
\newtheorem{lem}{Lemma}[section]

\theoremstyle{remark}

\allowdisplaybreaks[4]
\DeclareMathOperator{\td}{d\mspace{-2mu}}
\numberwithin{equation}{section}

\begin{document}

\title[A new proof of the geometric-arithmetic mean inequality]
{A new proof of the geometric-arithmetic mean inequality by Cauchy's integral formula}

\author[F. Qi]{Feng Qi}
\address[Qi]{Department of Mathematics, School of Science, Tianjin Polytechnic University, Tianjin City, 300387, China}
\email{\href{mailto: F. Qi <qifeng618@gmail.com>}{qifeng618@gmail.com}, \href{mailto: F. Qi <qifeng618@hotmail.com>}{qifeng618@hotmail.com}, \href{mailto: F. Qi <qifeng618@qq.com>}{qifeng618@qq.com}}
\urladdr{\url{http://qifeng618.wordpress.com}}

\author[X.-J. Zhang]{Xiao-Jing Zhang}
\address[Zhang]{Department of Mathematics, School of Science, Tianjin Polytechnic University, Tianjin City, 300387, China}
\email{\href{mailto: X.-J. Zhang <xiao.jing.zhang@qq.com>}{xiao.jing.zhang@qq.com}}

\author[W.-H. Li]{Wen-Hui Li}
\address[Li]{Department of Mathematics, School of Science, Tianjin Polytechnic University, Tianjin City, 300387, China}
\email{\href{mailto: W.-H. Li <wen.hui.li@foxmail.com>}{wen.hui.li@foxmail.com}}



\begin{abstract}
Let $a=(a_1,a_2,\dotsc,a_n)$ for $n\in\mathbb{N}$ be a given sequence of positive numbers.
In the paper, the authors establish, by using Cauchy's integral formula in the theory of complex functions, an integral representation of the principal branch of the geometric mean
\begin{equation*}
G_n(a+z)=\Biggl[\prod_{k=1}^n(a_k+z)\Biggr]^{1/n}
\end{equation*}
for $z\in\mathbb{C}\setminus(-\infty,-\min\{a_k,1\le k\le n\}]$, and then provide a new proof of the well known GA mean inequality.
\end{abstract}

\thanks{This paper was typeset using \AmS-\LaTeX}

\maketitle

\section{Introduction}

Let $a=(a_1,a_2,\dotsc,a_n)$ for $n\in\mathbb{N}$, the set of all positive integers, be a given sequence of positive numbers. Then the arithmetic and geometric means $A_n(a)$ and $G_n(a)$ of the numbers $a_1, a_2, \dotsc, a_n$ are defined respectively as
\begin{equation}
A_n(a)=\frac1n\sum_{k=1}^na_k
\end{equation}
and
\begin{equation}
G_n(a)=\Biggl(\prod_{k=1}^na_k\Biggr)^{1/n}.
\end{equation}
It is common knowledge that
\begin{equation}\label{AG-ineq}
G_n(a)\le A_n(a),
\end{equation}
with equality if and only if $a_1=a_2=\dotsm=a_n$.
\par
There has been a large number, presumably over one hundred, of proofs of the GA mean inequality~\eqref{AG-ineq} in the mathematical literature. The most complete information, so far, can be found in the monographs~\cite{bellman, bullenmean, hlp, kuang-3rd, mit, Mitrinovic-Vasic-1969} and a lot of references therein.
\par
In this paper, we establish, by using Cauchy's integral formula in the theory of complex functions, an integral representation of the principal branch of the geometric mean
\begin{equation}
G_n(a+z)=\Biggl[\prod_{k=1}^n(a_k+z)\Biggr]^{1/n}
\end{equation}
for $z\in\mathbb{C}\setminus(-\infty,-\min\{a_k,1\le k\le n\}]$, and provide a new proof of the GA mean inequality~\eqref{AG-ineq}.

\section{Lemmas}

In order to prove our main results, we need the following lemmas.

\begin{lem}\label{AG-New-lem=1}
For $z\in\mathbb{C}\setminus(-\infty,-\min\{a_k,1\le k\le n\}]$, the principal branch of the complex function
\begin{equation}
f_n(z)=G_n(a+z)-z,
\end{equation}
where $a+z=(a_1+z,a_2+z,\dotsc,a_n+z)$, meets
\begin{equation}\label{AG-New-lem=1-lim}
\lim_{z\to\infty}f_n(z)=A_n(a).
\end{equation}
\end{lem}

\begin{proof}
By L'H\^ospital's rule in the theory of complex functions, we have
\begin{multline*}
\lim_{z\to\infty}f_n(z)=\lim_{z\to\infty}\biggl\{z\biggl[G_n\biggl(1+\frac{a}z\biggr)-1\biggr]\biggr\} \\
=\lim_{z\to0}\frac{G_n(1+az)-1}{z}
=\lim_{z\to0}\frac{\td}{\td z}\Biggl[\prod_{k=1}^n(1+a_kz)\Biggr]^{1/n}
=A_n(a),
\end{multline*}
where
$
1+\frac{a}z=\bigl(1+\frac{a_1}z,1+\frac{a_2}z,\dotsc,1+\frac{a_n}z\bigr)
$
and
$1+az=(1+a_1z,1+a_2z,\dotsc,1+a_nz)$. Lemma~\ref{AG-New-lem=1} is thus proved.
\end{proof}

\begin{lem}\label{Im-comput-AG-New-P-lem}
For $z\in\mathbb{C}\setminus(-\infty,0]$ and $a=(a_1,a_2,\dotsc,a_n)$ satisfying $a_\ell\le a_{\ell+1}$ for $1\le \ell\le n-1$, let
\begin{equation}
h_n(z)=G_n(a-a_1+z)-z,
\end{equation}
where $a-a_1+z=(z,a_2-a_1+z,\dotsc,a_n-a_1+z)$. Then the imaginary part of the principal branch of $h_n(z)$ meets
\begin{equation}\label{Im-comput-AG-New-P-eq}
\lim_{\varepsilon\to0^+}\Im h_n(-t+i\varepsilon)=
\begin{cases}\displaystyle
\Biggl[\prod_{k=1}^n|a_k-a_1-t|\Biggr]^{1/n}\sin\frac{\ell\pi}n, & t\in(a_\ell-a_1,a_{\ell+1}-a_1]\\
0, & t>a_n-a_1
\end{cases}
\end{equation}
for $1\le \ell\le n-1$.
\end{lem}

\begin{proof}
For $t=a_{\ell+1}-a_1$ for $1\le\ell\le n-1$, we have
\begin{multline*}
h_n(-t+i\varepsilon)=\exp\Biggl[\frac1n\sum_{k\ne\ell+1}^n \ln(a_k-a_1-t+i\varepsilon)+\frac1n\ln(i\varepsilon)\Biggr]+t-i\varepsilon\\
\begin{aligned}
&=\exp\Biggl[\frac1n\sum_{k\ne\ell+1}^n \ln(a_k-a_1-t+i\varepsilon)\Biggr] \exp\biggl[\frac1n\biggl(\ln|\varepsilon|+\frac\pi2i\biggr)\biggl]+t-i\varepsilon\\
&\to\exp\Biggl[\frac1n\sum_{k\ne\ell+1}^n \ln(a_k-a_1-t)\Biggr] \lim_{\varepsilon\to0^+}\exp\biggl[\frac1n\biggl(\ln|\varepsilon|+\frac\pi2i\biggr)\biggl]+t\\
&=t
\end{aligned}
\end{multline*}
as $\varepsilon\to0^+$. Hence, when $t=a_{\ell+1}-a_1$ for $1\le\ell\le n-1$, we have
\begin{equation*}
\lim_{\varepsilon\to0^+}\Im h_n(-t+i\varepsilon)=0.
\end{equation*}
\par
For $t\in(0,\infty)\setminus\{a_{\ell+1}-a_1,1\le\ell\le n-1\}$ and $\varepsilon>0$, we have
\begin{gather*}
\begin{aligned}
h_n(-t+i\varepsilon)&=G_n(a-a_1-t+i\varepsilon)+t-i\varepsilon\\
&=\exp\Biggl[\frac1n\sum_{k=1}^n\ln(a_k-a_1-t+i\varepsilon)\Biggr]+t-i\varepsilon\\
&=\exp\Biggl\{\frac1n\sum_{k=1}^n[\ln|a_k-a_1-t+i\varepsilon| +i\arg(a_k-a_1-t+i\varepsilon)]\Biggr\}+t-i\varepsilon\\
&\to
\begin{cases}\displaystyle
\exp\Biggl(\frac1n\sum_{k=1}^n\ln|a_k-a_1-t| +\frac{\ell\pi}ni\Biggr)+t, & t\in(a_\ell-a_1,a_{\ell+1}-a_1)\\ \displaystyle
\exp\Biggl(\frac1n\sum_{k=1}^n\ln|a_k-a_1-t| +\pi i\Biggr)+t, & t>a_n-a_1
\end{cases}
\end{aligned}\\
=
\begin{cases}\displaystyle
\Biggl(\prod_{k=1}^n|a_k-a_1-t|\Biggr)^{1/n} \biggl(\cos\frac{\ell\pi}n+i\sin\frac{\ell\pi}n\biggr)+t, & t\in(a_\ell-a_1,a_{\ell+1}-a_1)\\ \displaystyle
\Biggl(\prod_{k=1}^n|a_k-a_1-t|\Biggr)^{1/n} (\cos\pi+i\sin\pi)+t, & t>a_n-a_1
\end{cases}
\end{gather*}
as $\varepsilon\to0^+$. As a result, we have
\begin{equation*}
\lim_{\varepsilon\to0^+}\Im h_n(-t+i\varepsilon)=
\begin{cases}\displaystyle
\Biggl(\prod_{k=1}^n|a_k-a_1-t|\Biggr)^{1/n}\sin\frac{\ell\pi}n, & t\in(a_\ell-a_1,a_{\ell+1}-a_1);\\
0, & t>a_n-a_1.
\end{cases}
\end{equation*}
The proof of Lemma~\ref{Im-comput-AG-New-P-lem} is completed.
\end{proof}

\section{An integral representation of the geometric mean}

Now we are in a position to establish an integral representation of the geometric mean $G_n(a+z)$.

\begin{thm}\label{AG-New-thm1}
Let $0<a_k\le a_{k+1}$ for $1\le k\le n-1$ and $a+z=(a_1+z,a_2+z,\dotsc,a_n+z)$ for $z\in\mathbb{C}\setminus(-\infty,-a_1]$. Then the principal branch of the geometric mean $G_n(a+z)$ has the integral representation
\begin{equation}\label{AG-New-eq1}
G_n(a+z)=A_n(a)+z-\frac1\pi\sum_{\ell=1}^{n-1}\sin\frac{\ell\pi}n \int_{a_\ell}^{a_{\ell+1}} \Biggl|\prod_{k=1}^n(a_k-t)\Biggr|^{1/n} \frac{\td t}{t+z}.
\end{equation}
\end{thm}

\begin{proof}
By standard arguments, it is not difficult to see that
\begin{equation}\label{weighted-geometric-eq2-n}
\lim_{z\to0^+}[zh_n(z)]=0\quad \text{and}\quad h_n(\overline{z})=\overline{h_n(z)}.
\end{equation}
\par
For any but fixed point $z\in\mathbb{C}\setminus(-\infty,0]$, choose $0<\varepsilon<1$ and $r>0$ such that $0<\varepsilon<|z|<r$, and consider the positively oriented contour $C(\varepsilon,r)$ in $\mathbb{C}\setminus(-\infty,0]$ consisting of the half circle $z=\varepsilon e^{i\theta}$ for $\theta\in\bigl[-\frac\pi2,\frac\pi2\bigr]$ and the half lines $z=x\pm i\varepsilon$ for $x\le0$ until they cut the circle $|z|=r$, which close the contour at the points $-r(\varepsilon)\pm i\varepsilon$, where $0<r(\varepsilon)\to r$ as $\varepsilon\to0$.
By the famous Cauchy's integral formula in the theory of complex functions, we have
\begin{equation}
\begin{split}\label{h(z)-Cauchy-Apply}
h_n(z)&=\frac1{2\pi i}\oint_{C(\varepsilon,r)}\frac{h_n(w)}{w-z}\td w\\
&=\frac1{2\pi i}\biggl[\int_{\pi/2}^{-\pi/2}\frac{i\varepsilon e^{i\theta}h\bigl(\varepsilon e^{i\theta}\bigr)}{\varepsilon e^{i\theta}-z}\td\theta  +\int_{\arg[-r(\varepsilon)-i\varepsilon]}^{\arg[-r(\varepsilon)+i\varepsilon]}\frac{ir e^{i\theta}h\bigl(re^{i\theta}\bigr)}{re^{i\theta}-z}\td\theta\\
&\quad+\int_{-r(\varepsilon)}^0 \frac{h_n(x+i\varepsilon)}{x+i\varepsilon-z}\td x+\int_0^{-r(\varepsilon)}\frac{h_n(x-i\varepsilon)}{x-i\varepsilon-z}\td x\biggr].
\end{split}
\end{equation}
By the limit in~\eqref{weighted-geometric-eq2-n}, it follows that
\begin{equation}\label{zf(z)=0}
\lim_{\varepsilon\to0^+}\int_{\pi/2}^{-\pi/2}\frac{i\varepsilon e^{i\theta}h_n\bigl(\varepsilon e^{i\theta}\bigr)}{\varepsilon e^{i\theta}-z}\td\theta=0.
\end{equation}
By virtue of the limit~\eqref{AG-New-lem=1-lim} in Lemma~\ref{AG-New-lem=1}, we deduce that
\begin{equation}\label{big-circle-int=0}
\begin{split}
\lim_{\substack{\varepsilon\to0^+\\r\to\infty}} \int_{\arg[-r(\varepsilon)-i\varepsilon]}^{\arg[-r(\varepsilon)+i\varepsilon]}\frac{ir e^{i\theta}h_n\bigl(re^{i\theta}\bigr)}{re^{i\theta}-z}\td\theta
&=\lim_{r\to\infty}\int_{-\pi}^{\pi}\frac{ir e^{i\theta}h_n\bigl(re^{i\theta}\bigr)}{re^{i\theta}-z}\td\theta\\
&=2A_n(a-a_1)\pi i,
\end{split}
\end{equation}
where $a-a_1=(0,a_2-a_1,\dotsc,a_n-a_1)$.
Utilizing the second formula in~\eqref{weighted-geometric-eq2-n} and the limit~\eqref{Im-comput-AG-New-P-eq} in Lemma~\ref{Im-comput-AG-New-P-lem} results in
\begin{align}
&\quad\int_{-r(\varepsilon)}^0 \frac{h_n(x+i\varepsilon)}{x+i\varepsilon-z}\td x
+\int_0^{-r(\varepsilon)}\frac{h_n(x-i\varepsilon)}{x-i\varepsilon-z}\td x \notag\\
&=\int_{-r(\varepsilon)}^0 \biggl[\frac{h_n(x+i\varepsilon)}{x+i\varepsilon-z}-\frac{h_n(x-i\varepsilon)}{x-i\varepsilon-z}\biggr]\td x\notag\\
&=\int_{-r(\varepsilon)}^0\frac{(x-i\varepsilon-z)h_n(x+i\varepsilon) -(x+i\varepsilon-z)h_n(x-i\varepsilon)} {(x+i\varepsilon-z)(x-i\varepsilon-z)}\td x\notag\\
&=\int_{-r(\varepsilon)}^0\frac{(x-z)[h_n(x+i\varepsilon)-h_n(x-i\varepsilon)] -i\varepsilon[h_n(x-i\varepsilon)+h_n(x+i\varepsilon)]} {(x+i\varepsilon-z)(x-i\varepsilon-z)}\td x\notag\\
&=2i\int_{-r(\varepsilon)}^0\frac{(x-z)\Im h_n(x+i\varepsilon) -\varepsilon\Re h_n(x+i\varepsilon)} {(x+i\varepsilon-z)(x-i\varepsilon-z)}\td x\notag\\
&\to2i\int_{-r}^0\frac{\lim_{\varepsilon\to0^+}\Im h_n(x+i\varepsilon)}{x-z}\td x\notag\\
&=-2i\int^r_0\frac{\lim_{\varepsilon\to0^+}\Im h_n(-t+i\varepsilon)}{t+z}\td t\notag\\
&\to-2i\int^\infty_0\frac{\lim_{\varepsilon\to0^+}\Im h_n(-t+i\varepsilon)}{t+z}\td t\notag\\
&=-2i\sum_{\ell=1}^{n-1}\sin\frac{\ell\pi}n \int_{a_\ell-a_1}^{a_{\ell+1}-a_1} \Biggl[\prod_{k=1}^n|a_k-a_1-t|\Biggr]^{1/n} \frac{\td t}{t+z} \label{level0lines}
\end{align}
as $\varepsilon\to0^+$ and $r\to\infty$. Substituting equations~\eqref{zf(z)=0}, \eqref{big-circle-int=0}, and~\eqref{level0lines} into~\eqref{h(z)-Cauchy-Apply} and simplifying generate
\begin{equation}\label{h-n(z)=int}
h_n(z)=A_n(a-a_1)-\frac1\pi\sum_{\ell=1}^{n-1}\sin\frac{\ell\pi}n \int_{a_\ell-a_1}^{a_{\ell+1}-a_1} \Biggl[\prod_{k=1}^n|a_k-a_1-t|\Biggr]^{1/n} \frac{\td t}{t+z}.
\end{equation}
From $f_n(z)=h_n(z+a_1)+a_1$ and~\eqref{h-n(z)=int}, it is immediate to deduce that
\begin{align*}
f_n(z)&=A_n(a-a_1)+a_1-\frac1\pi\sum_{\ell=1}^{n-1}\sin\frac{\ell\pi}n \int_{a_\ell-a_1}^{a_{\ell+1}-a_1} \Biggl[\prod_{k=1}^n|a_k-a_1-t|\Biggr]^{1/n} \frac{\td t}{t+z+a_1}\\
&=A_n(a)-\frac1\pi\sum_{\ell=1}^{n-1}\sin\frac{\ell\pi}n \int_{a_\ell}^{a_{\ell+1}} \Biggl[\prod_{k=1}^n|a_k-t|\Biggr]^{1/n} \frac{\td t}{t+z},
\end{align*}
from which the integral representation~\eqref{AG-New-eq1} follows. Theorem~\ref{AG-New-thm1} is proved.
\end{proof}

\section{A new proof of the GA mean inequality}

With the aid of the integral representation~\eqref{AG-New-eq1} in Theorem~\ref{AG-New-thm1}, we can easily deduce the GA mean inequality~\eqref{AG-ineq} as follows.
\par
Taking $z=0$ in the integral representation~\eqref{AG-New-eq1} yields
\begin{equation*}
G_n(a)=A_n(a)-\frac1\pi\sum_{\ell=1}^{n-1}\sin\frac{\ell\pi}n \int_{a_\ell}^{a_{\ell+1}} \Biggl[\prod_{k=1}^n|a_k-t|\Biggr]^{1/n} \frac{\td t}{t}\le A_n(a).
\end{equation*}
From this, it is also obvious that the equality in~\eqref{AG-ineq} is valid if and only if $a_1=a_2=\dotsm=a_n$. The proof of the GA mean inequality~\eqref{AG-ineq} is complete.

\end{document}